\documentclass[11pt]{amsart}
\usepackage[psamsfonts]{amssymb}
\usepackage{amsthm,amsmath}

\title{Polynomial functions over bounded distributive lattices}

\author{Miguel Couceiro}
\address{Mathematics Research Unit, FSTC, University of Luxembourg, 6, rue Coudenhove-Kalergi, L-1359 Luxembourg-Kirchberg, Luxembourg.} \email{miguel.couceiro[at]uni.lu}

\author{Jean-Luc Marichal}
\address{Mathematics Research Unit, FSTC, University of Luxembourg, 6, rue Coudenhove-Kalergi, L-1359 Luxembourg-Kirchberg, Luxembourg.} \email{jean-luc.marichal[at]uni.lu }

\date{December 2, 2009}

\begin{document}
\maketitle

\theoremstyle{plain}
\newtheorem{theorem}{Theorem}
\newtheorem{lemma}[theorem]{Lemma}
\newtheorem{proposition}[theorem]{Proposition}
\newtheorem{corollary}[theorem]{Corollary}
\newtheorem{fact}[theorem]{Fact}
\newtheorem*{main}{Main Theorem}

\theoremstyle{definition}
\newtheorem{definition}[theorem]{Definition}
\newtheorem{example}[theorem]{Example}

\theoremstyle{remark}
\newtheorem{claim}{Claim}

\newcommand{\card}[1]{\ensuremath{\lvert{#1}\rvert}}
\newcommand{\vect}[1]{\ensuremath{\mathbf{#1}}} 
\newcommand{\co}[1]{\ensuremath{\overline{#1}}}
\newcommand{\ran}{\mathcal{R}}
\def\med{\mathop{\rm med}\nolimits}

\begin{abstract}
Let $L$ be a bounded distributive lattice. We give several characterizations of those $L^n \to L$ mappings that are polynomial functions, i.e.,
functions which can be obtained from projections and  constant functions using binary joins and meets. Moreover, we discuss the disjunctive
normal form representations of these polynomial functions.
\end{abstract}


\section{Introduction}

Let $(L;\wedge,\vee)$ be a lattice. With no danger of ambiguity, we denote lattices by their universes. By a (\emph{lattice}) \emph{polynomial function} we simply mean a map  $f\colon L^n \to L$ which can be obtained by
composition of the binary operations $\wedge$ and $\vee$, the projections, and the constant functions; see, e.g., page 93 in \cite{BurSan81}. If
constant functions are not used, then these polynomial functions are usually referred to as \emph{term functions}. For general background, see \cite{BurSan81,Denecke,Denecke2,Grae03}

For a finite lattice $L$, the set of all polynomial functions on $L$ is well understood. Indeed, Kindermann \cite{Kind78} reduces the problem of
describing polynomial functions to tolerances, and reasonable descriptions for the latter have been provided in Czedli and Klukovits
\cite{CzeKlu83} and Chajda \cite{Chajda}.

The goal of the current paper is to present a more direct approach to polynomial functions and provide alternative descriptions, different in
nature and flavor, in the case when $L$ is distributive, with 0 and 1 as bottom and top elements. Notably enough, instead of finiteness it
suffices to assume that $L$ is a bounded distributive lattice. So, throughout the paper,
we assume that $L$ is a bounded distributive lattice.
 Also, functions that are not order-preserving cannot be polynomial functions. Thus, our
main result focuses on order-preserving functions.

We shall make use of the following notation. The ternary median term $(x\vee y)\wedge (x\vee z)\wedge (y\vee z )$ will be denoted by
$\med(x,y,z)$. For $c\in L$, the constant tuple $(c, \ldots , c)$ in $L^n$ will be denoted by $\overline{c}$. For $k\in [n]=\{1,\ldots,n\}$, $a\in L$,
and $\vect{x}=(x_1,\ldots,x_n)\in L^n$, let $\vect{x}_k^a$ be the tuple in $L^n$ whose $i$th component is $a$, if $i=k$, and $x_i$, otherwise.
Let $[\vect{x}]_c$ (resp.\ $[\vect{x}]^c$) denote the $n$-tuple whose $i$th component is $0$ (resp.\ $1$), if $x_i\leqslant c$ (resp.\
$x_i\geqslant c$), and $x_i$, otherwise.

The range of a function $f\colon L^{n}\rightarrow L$ is defined by $\ran(f)=\{f(\vect{x}) : x\in L^n\}$. For $k\in [n]$ and
$\vect{a}=(a_1,\ldots,a_n)\in L^n$, we also define the unary function $f_{-k}^{\vect{a}}\colon L\to L$ as
$f_{-k}^{\vect{a}}(x)=f(\vect{a}_k^x)$.

Note that if $f\colon L^{n}\rightarrow L$ is a polynomial function, then every unary function $g$ obtained from $f$ by substituting constants for $n-m$ of its variables and identifying the remaining $m$ variables is also a polynomial function. In fact, it is not difficult to see that such a function $g$ is of the form  $\med (a,x,b)$, for some $a,b\in L$, and thus every such function $g$ preserves $\wedge$ and $\vee$. In particular, we have the following fact.

\begin{fact}\label{fact1}
Every polynomial function $f\colon L^n\to L$ satisfies condition (\ref{eq:deltaFG}) and its dual, where
\begin{equation}\label{eq:deltaFG}
f(\vect{x}_k^{a\wedge b})=f(\vect{x}_k^a)\wedge f(\vect{x}_k^b)\quad \mbox{for all }\vect{x}\in L^n,~a,b\in L,\mbox{ and }k\in [n].
\end{equation}
Moreover, every function $f\colon L^n\to L$ satisfying (\ref{eq:deltaFG}) or its dual is order-preserving.
\end{fact}

The following result reassembles the various characterizations of polynomial functions provided in this paper, and its proof is given in Section
2.
\begin{main} \label{mainChar}
Let $L$ be a bounded distributive lattice and $f\colon L^{n}\rightarrow L$, $n\geqslant 1$, be an order-preserving function. The following
conditions are equivalent:
\begin{enumerate}
\item[$(i)$] $f$ is a polynomial function;

\item[$(ii)$] for every $\vect{x} \in L^n$ and $k\in [n]$,
\begin{align}\label{MedDecomposition}
 f(\vect{x})=\med\big(f(\vect{x}^{0}_{k}), x_k, f(\vect{x}^{1}_{k})\big);
\end{align}

\item[$(iii)$] $f$ satisfies (\ref{eq:deltaFG}) and its dual, the sets $\ran(f)$ and $\ran(f_{-k}^{\vect{a}})$, for every $\vect{a}\in L^n$ and
$k\in [n]$, are convex, and for every $\vect{x} \in L^n$ and $k\in [n]$,
\begin{equation}\label{eq:StrId}
f(\vect{x}^{f(\vect{x})}_{k})=f(\vect{x});
\end{equation}

\item[$(iv)$]  $f$ satisfies condition (\ref{eq:homogeneity}) and its dual, where
\begin{equation}\label{eq:homogeneity}
f(\vect{x}\wedge \overline{c}) = f(\vect{x})\wedge c \quad \mbox{for all }c\in [f(\overline{0}),f(\overline{1})];
\end{equation}

\item[$(v)$] $f$ satisfies the dual of (\ref{eq:deltaFG}), and (\ref{eq:homogeneity}), and
\begin{equation}\label{eq:horizontal}
f(\vect{x}) = f(\vect{x}\wedge \overline{c})\vee f([\vect{x}]_c)\quad \mbox{for all }c\in [f(\overline{0}),f(\overline{1})];
\end{equation}

\item[$(vi)$] $f$ satisfies (\ref{eq:deltaFG}) and its dual, (\ref{eq:horizontal}) and its dual, and
\begin{equation}\label{eq:idempotency}
f(\overline{c}) ={c}\quad \mbox{for all }c\in [f(\overline{0}),f(\overline{1})].
\end{equation}\end{enumerate}
\end{main}

Even though not evident at the first sight, note that from the Main Theorem it follows that $(v)$ is equivalent to its dual. Note also that
every function satisfying (\ref{MedDecomposition}) is order-preserving. The equivalence between $(i)$ and $(ii)$ was first established in
\cite[Theorem 17]{Marc}. Moreover, by Fact~\ref{fact1}, it follows that the order-preservation condition is redundant for all except assertion
$(iv)$.

Let $2^{[n]}$ denote the set of all subsets of $[n]$. If $\alpha \colon 2^{[n]}\to L$ is a mapping, then
\begin{equation}\label{eq:DNF}
\bigvee_{I\subseteq [n]}\big(\alpha(I)\wedge \bigwedge_{i\in I} x_i\big)
\end{equation}
is called a \emph{disjunctive normal form} over $L$. For a function $f\colon L^{n}\rightarrow L$, let $\mathrm{DNF}(f)$ denote
 the set of those maps $\alpha \colon 2^{[n]}\to L$ for which (\ref{eq:DNF}), as an $L^n\to L$ mapping, coincides with $f$. Observe that $\mathrm{DNF}(f)=\varnothing$ if $f$ is \emph{not} a polynomial function. For $I \subseteq [n]$, let $\vect{e}_I$ be the
\emph{characteristic vector} of $I$, i.e., the tuple in $L^n$ whose $i$th component is $1$ if $i \in I$, and 0 otherwise. Define $\alpha_f
\colon 2^{[n]}\to L$, $I\mapsto f (\vect{e}_I )$.

\begin{lemma}[{Goodstein \cite{Goo67}}]\label{lemma:DNF(f)}
 If $L$ is a bounded distributive lattice and $f \colon L^n \to L$
 a polynomial function, then $\alpha_f\in\mathrm{DNF}(f)$. In particular, each polynomial function has a
disjunctive normal form representation.
\end{lemma}

By Lemma \ref{lemma:DNF(f)}, for each polynomial function $f \colon L^n \to L$, we have that
\begin{equation}\label{Goodstein}
\mbox{$f$ is uniquely determined by its restriction to $\{0, 1\}^n$.}
\end{equation}
It is noteworthy that, by (\ref{Goodstein}), term functions are exactly those polynomial functions $f:L^n\to L$ for which $\{0\}$, $\{1\}$, and
$\{0,1\}$ constitute subalgebras of $(L,f)$. In addition to the Main Theorem, we prove the following result strengthening Lemma
\ref{lemma:DNF(f)}.

\begin{proposition}\label{prop:DNF} Let $L$ be a bounded distributive lattice, $f \colon L^n \to L$
 a polynomial function, and  $\alpha \colon 2^{[n]}\to L$ a mapping. Then $\alpha \in\mathrm{DNF}(f)$ if and only if $\bigvee_{J \subseteq I}\alpha (J)= \alpha_f(I)$ for all $I \subseteq [n]$.
 \end{proposition}

Using Proposition \ref{prop:DNF}, it is straightforward to construct examples of lattices $L$ and polynomial functions $f \colon L^n \to L$ for
which $|\mathrm{DNF}(f)| > 1$, and to provide some technical conditions characterizing those polynomial functions $f$ for which
$|\mathrm{DNF}(f)| = 1$. The trivial details are left to the reader.

\section{Technicalities and proofs}

In this section we provide the proofs of the Main Theorem and Proposition \ref{prop:DNF}. First, we prove the latter.

\begin{proof}[\bf Proof of Proposition \ref{prop:DNF}]
Let $L$ be a bounded distributive lattice, $f \colon L^n \to L$ a polynomial function, and  $\alpha \colon 2^{[n]}\to L$ a mapping.

 Suppose first that $\alpha \in\mathrm{DNF}(f)$. Then, for every $I \subseteq [n]$, $\alpha_f(I)=f(\vect{e}_I)=\bigvee_{J \subseteq I}\alpha (J)$.  Now suppose that $\bigvee_{J \subseteq I}\alpha (J)= \alpha_f(I)$, for all $I \subseteq [n]$, and let $g \colon L^n \to L$ be the polynomial
 function such that $\alpha \in\mathrm{DNF}(g)$. Clearly,  for every $I \subseteq [n]$,  we have $g(\vect{e}_I)=\bigvee_{J \subseteq I}\alpha (J)=\alpha_f(I)=f(\vect{e}_I)$. From (\ref{Goodstein}) it follows that $g=f$ and hence, $\alpha \in\mathrm{DNF}(f)$.
 \end{proof}

To prove the Main Theorem we will need some auxiliary results. We proceed by focusing first on the conditions given in $(iv)$, $(v)$, and
$(vi)$.

\begin{lemma}\label{wlpHomogeneous}
Every polynomial function $f:L^n\to L$ satisfies (\ref{eq:homogeneity}) and its dual.
\end{lemma}

\begin{proof}
Let $f\colon L^{n}\rightarrow L$ be a polynomial function. For any $c\in[f(\overline{0}),f(\overline{1})]$, we have
\begin{eqnarray*}
f(\vect{x}\wedge \overline{c}) &=& \bigvee_{I\subseteq [n]}\Big(\alpha_f(I)\wedge \bigwedge_{i\in I} (x_i\wedge c)\Big)%
~=~ f(\overline{0})\vee \bigvee_{\textstyle{I\subseteq [n]\atop I\neq\varnothing}}\Big(\alpha_f(I)\wedge \bigwedge_{i\in I} (x_i\wedge c)\Big)\\
&=& \Big(f(\overline{0})\vee \bigvee_{\textstyle{I\subseteq [n]\atop I\neq\varnothing}}\big(\alpha_f(I)\wedge \bigwedge_{i\in I}
x_i\big)\Big)\wedge c ~=~ f(\vect{x})\wedge c.
\end{eqnarray*}
Similarly, it follows that $f$ satisfies the dual of (\ref{eq:homogeneity}).
\end{proof}

\begin{lemma}\label{lemma:MinMaxIdem}
Let $f\colon L^{n}\rightarrow L$ be an order-preserving function. If $f$ satisfies (\ref{eq:homogeneity}) or its dual, then it satisfies
(\ref{eq:idempotency}). In particular, $\ran(f)$ coincides with $[f(\overline{0}),f(\overline{1})]$ and thus is convex.
\end{lemma}

\begin{proof}
If $f$ satisfies (\ref{eq:homogeneity}), then for any $c\in [f(\overline{0}),f(\overline{1})]$, we have $f(\overline{c})=f(\overline{1}\wedge
\overline{c})=f(\overline{1})\wedge c=c$ and thus $f$ satisfies (\ref{eq:idempotency}). The dual statement follows similarly. The last claim
follows immediately from (\ref{eq:idempotency}).
\end{proof}

\begin{lemma}\label{lemma:new}
Let $f\colon L\to L$ be an order-preserving function. If $f$ satisfies (\ref{eq:homogeneity}) and its dual, then it preserves $\wedge$ and
$\vee$.
\end{lemma}

\begin{proof}
If $f\colon L\to L$ satisfies (\ref{eq:homogeneity}) and its dual, by Lemma~\ref{lemma:MinMaxIdem}, it satisfies (\ref{eq:idempotency}). Then,
for every $x\in L$, we have
\begin{eqnarray*}
f(x)&=&\med(f(0),f(x),f(1))=f(\med(f(0),x,f(1)))\\
&=&\med(f(0),x,f(1)).
\end{eqnarray*}
It is then immediate to see that $f$ preserves $\wedge$ and $\vee$.
 \end{proof}

\begin{lemma}\label{Hor-Min-Hom}
 Let $f\colon L^{n}\rightarrow L$ be an order-preserving function. If $f$ satisfies  (\ref{eq:homogeneity}) and its dual, then it satisfies (\ref{eq:horizontal}) and its
 dual. Moreover, $f$ satisfies (\ref{eq:deltaFG}) and its dual.
\end{lemma}

\begin{proof}
Suppose that $f$ satisfies (\ref{eq:homogeneity}) and its dual.  For any $\vect{x}\in L^n$ and any $c\in [f(\overline{0}),f(\overline{1})]$, we
have
\begin{eqnarray*}
  f(\vect{x}\wedge \overline{c})\vee f([\vect{x}]_c) &=& \big(f(\vect{x})\wedge c\big)\vee f([\vect{x}]_c)\\
  &=& \big(f(\vect{x})\vee f([\vect{x}]_c)\big)\wedge\big(c\vee f([\vect{x}]_c)\big) \\
  &=& f(\vect{x})\wedge f(\overline{c}\vee [\vect{x}]_c)%
  ~=~ f(\vect{x})
\end{eqnarray*}
and hence, $f$ satisfies (\ref{eq:horizontal}). The dual statement follows similarly.

\begin{claim}\label{cl:1}
For every $k\in [n]$ and $\vect{a}\in L^n$, the unary function $f_{-k}^{\vect{a}}$ satisfies (\ref{eq:homogeneity}) and its dual.
\end{claim}

\begin{proof}[Proof of Claim~\ref{cl:1}]
Let $k\in [n]$, $\vect{a}\in L^n$, and consider $c\in [f_{-k}^{\vect{a}}(0),f_{-k}^{\vect{a}}(1)]$. By (\ref{eq:horizontal}),
(\ref{eq:homogeneity}), and the dual of (\ref{eq:homogeneity}), for $x\in L$, we have
\begin{eqnarray*}
f_{-k}^{\vect{a}}(x\vee c) &=& f_{-k}^{\vect{a}\wedge \overline{c}}(c)\vee f_{-k}^{[\vect{a}]_c}([x\vee c]_c) ~=~ \big(f_{-k}^{\vect{a}}(1)\wedge c\big)\vee f_{-k}^{[\vect{a}]_c}([x\vee c]_c)\\
&=& c\vee f_{-k}^{[\vect{a}]_c}([x\vee c]_c) ~=~ f_{-k}^{\vect{a}\vee\overline{c}}(x\vee c)\\
&=& f_{-k}^{\vect{a}}(x)\vee c.
\end{eqnarray*}
Therefore $f_{-k}^{\vect{a}}$ satisfies the dual of (\ref{eq:homogeneity}). Similarly, we can prove that it also satisfies
(\ref{eq:homogeneity}).
\end{proof}

By Claim~\ref{cl:1}, each function $f_{-k}^{\vect{a}}$ satisfies (\ref{eq:homogeneity}) and its dual. By Lemma~\ref{lemma:new}, it preserves
$\wedge$ and $\vee$.
\end{proof}

As mentioned, the equivalence between $(i)$ and $(ii)$ in the Main Theorem was shown in \cite[Theorem 17]{Marc}. For the sake of
self-containment, we provide a simpler proof here.

\begin{proposition}[{\cite{Marc}}]\label{Theorem:MAR}
A function $f\colon L^{n}\rightarrow L$ is a polynomial function if and only if it satisfies (\ref{MedDecomposition}) for every $\vect{x} \in
L^n$ and $k\in [n]$.
\end{proposition}

\begin{proof}
On the one hand, if $f\colon L^{n}\rightarrow L$ is a polynomial function, then, for every $\vect{x}\in L^n$ and $k\in [n]$, we have
$$
f(\vect{x}_k^0) = \bigvee_{\textstyle{I\subseteq [n]\atop I\not\ni k}}\Big(\alpha_f(I)\wedge \bigwedge_{i\in I} x_i\Big)
$$
and
\begin{eqnarray*}
f(\vect{x}_k^1) &=& \bigvee_{\textstyle{I\subseteq [n]\atop I\not\ni k}}\Big(\alpha_f(I)\wedge \bigwedge_{i\in I} x_i\Big)\vee \bigvee_{\textstyle{I\subseteq [n]\atop I\ni k}}\Big(\alpha_f(I)\wedge \bigwedge_{i\in I\setminus\{k\}} x_i\Big)\\
&=& f(\vect{x}_k^0)\vee \bigvee_{\textstyle{I\subseteq [n]\atop I\ni k}}\Big(\alpha_f(I)\wedge \bigwedge_{i\in I\setminus\{k\}} x_i\Big)
\end{eqnarray*}
and hence
$$
\med\big(f(\vect{x}^{0}_{k}), x_k, f(\vect{x}^{1}_{k})\big)=f(\vect{x}_k^0)\vee\big(x_k\wedge f(\vect{x}_k^1)\big)=f(\vect{x}),
$$
which shows that $f$ satisfies (\ref{MedDecomposition}) for every $\vect{x} \in
L^n$ and $k\in [n]$.

%

On the other hand, any function obtained from a function in this class by substituting constants for variables, is also in the class. Thus, if a  function $f\colon L^{n}\rightarrow L$ satisfies (\ref{MedDecomposition}) for every $\vect{x} \in L^n$ and $k\in [n]$, then by repeated applications of  (\ref{MedDecomposition}), we can easily verify that $f$ can be obtained by composition of the binary operations $\vee$  and
$\wedge$, the projections, and the constant functions. That is, $f$ is a polynomial function.
\end{proof}

Now we focus on the conditions given in $(iii)$, $(v)$, and $(vi)$ of the Main Theorem. We shall make use of the following general result.

\begin{lemma}\label{lemma:strSIdem}
Let $C$ be a class of functions $f\colon L^n\to L$ $(n\geqslant 1)$ such that
\begin{enumerate}
\item[(i)] the unary members of $C$ are polynomial functions;

\item[(ii)] for $n>1$, any unary function obtained from an $n$-ary function $f$ in $C$ by substituting constants for $n-1$ variables of $f$ is
also in $C$.
\end{enumerate}
Then $C$ is a class of polynomial functions.
\end{lemma}

\begin{proof}
Let $C$ be a class of functions satisfying the conditions of the lemma. We show that each $f\colon L^n\to L$ in $C$ is a polynomial function. By
condition $(i)$, the claim holds for $n=1$. So suppose that $n>1$. By Proposition~\ref{Theorem:MAR}, it is enough to show that $f$ satisfies
(\ref{MedDecomposition}) for every $\vect{x} \in L^n$ and $k\in [n]$. So let $\vect{a}\in L^n$ and $k\in [n]$. By condition $(ii)$, we have that $f_{-k}^{\vect{a}}\in C$, and hence
$f_{-k}^{\vect{a}}$ is a polynomial function. By Proposition~\ref{Theorem:MAR}, $f_{-k}^{\vect{a}}$  satisfies (ii) of the Main Theorem,  and
hence,
$$
f(\vect{a})=f_{-k}^{\vect{a}}(a_k)=\med\big(f_{-k}^{\vect{a}}(0),a_k,f_{-k}^{\vect{a}}(1)\big)=\med\big(f(\vect{a}_{k}^0),
a_k,f(\vect{a}_{k}^1)\big).
$$
Since the above holds for every $\vect{a}\in L^n$ and $k\in [n]$, it follows that $f$  satisfies (ii) of the Main Theorem, and thus it is a
polynomial function.
\end{proof}

Note that, for $n=1$, (\ref{eq:StrId}) reduces to the well-known \emph{idempotency equation} $f\circ f=f$; see for instance Kuczma et
al.~\cite[\S{11.9E}]{KucChoGer90}.

\begin{proposition}\label{prop:IdEq}
A unary function $f\colon L\to L$ is a polynomial function if and only if $\ran(f)$ is convex and $f$ is a solution of the idempotency equation
that preserves $\wedge$ and $\vee$.
\end{proposition}

\begin{proof}
By Proposition~\ref{Theorem:MAR}, every unary polynomial function $f\colon L\to L$ is of the form $f(x)=\med(f(0),x,f(1))$ and thus satisfies
the conditions stated in the proposition.

Conversely, let $f\colon L\to L$ be a solution of the idempotency equation that preserves $\wedge$ and $\vee$ and such that $\ran(f)$ is convex,
and let $x\in L$. If $x\in [f(0),f(1)]=\ran(f)$, then there is $z\in L$ such that $x=f(z)$ and hence
$$
f(x)=f(f(z))=f(z)=x=\med(f(0),x,f(1)).
$$
Otherwise, let
$z=\med(f(0),x,f(1))\in  [f(0),f(1)]$. Then, since $f$ preserves $\wedge$ and $\vee$, it is order-preserving, and we have
\begin{eqnarray*}
f(x)&=&\med\big (f(0),f(x),f(1)\big )=\med\big(f(f(0)),f(x),f(f(1))\big)\\
&=&f\big(\med(f(0),x,f(1))\big)=f(z)=z=\med(f(0),x,f(1)),
\end{eqnarray*}
which shows that $f$ is a polynomial function.
\end{proof}

\begin{proposition}\label{cor:StrIdemWLP}
An order-preserving function $f\colon L^{n}\rightarrow L$ is a polynomial function if and only if $f$ satisfies (\ref{eq:deltaFG}), its dual,
(\ref{eq:StrId}), and the sets $\ran(f)$ and $\ran(f_{-k}^{\vect{a}})$, for every $\vect{a}\in L^n$ and $k\in [n]$, are convex.
\end{proposition}

\begin{proof}
By Fact~\ref{fact1} and Lemmas~\ref{wlpHomogeneous} and \ref{lemma:MinMaxIdem}, if $f$ is a polynomial function, then $f$ satisfies
(\ref{eq:deltaFG}), its dual, and $\ran(f)$ is convex. By Proposition~\ref{Theorem:MAR}, every polynomial function $f$ satisfies (\ref{MedDecomposition}) (and hence (\ref{eq:StrId})) and each set $\ran(f_{-k}^{\vect{a}})$, for $\vect{a}\in L^n$ and $k\in [n]$, is convex.

To prove the converse claim, consider the class $C$ of order-preserving functions $f\colon L^n\to L$ $(n\geqslant 1)$ satisfying the conditions
of the proposition. Clearly, $C$ satisfies condition $(ii)$ of Lemma~\ref{lemma:strSIdem}. By Proposition~\ref{prop:IdEq}, $C$ satisfies also
condition $(i)$ of Lemma~\ref{lemma:strSIdem}, and hence, $C$ is a class of polynomial functions.
\end{proof}

\begin{proposition}\label{prop:unaryHorizontalPoly}
Let $f\colon L\to L$ be a function. The following conditions are equivalent:
\begin{enumerate}
\item[(i)] $f$ is a polynomial function;

\item[(ii)] $f$ satisfies (\ref{eq:homogeneity}) and preserves $\vee$;

\item[(iii)] $f$ satisfies (\ref{eq:idempotency}) and preserves $\wedge$ and $\vee$.
\end{enumerate}
\end{proposition}

\begin{proof}
The implication $(i)\Rightarrow (iii)$ follows from Lemmas~\ref{wlpHomogeneous}, \ref{lemma:MinMaxIdem}, and \ref{lemma:new}. The implication $(iii)\Rightarrow (ii)$
follows from the fact that if $f$ satisfies (\ref{eq:idempotency}) and preserves $\wedge$ then it satisfies (\ref{eq:homogeneity}). Finally, to
see that the implication $(ii)\Rightarrow (i)$ holds, observe first that (\ref{eq:homogeneity}) implies (\ref{eq:idempotency}) by
Lemma~\ref{lemma:MinMaxIdem}. Since $f$ preserves $\vee$, we have that $f$ satisfies the dual of (\ref{eq:homogeneity}). Moreover, $f$ is
clearly order-preserving, and we have
\begin{eqnarray*}
f(x) &=& \med\big(f(0),f(x),f(1)\big)=f(0)\vee\big(f(x)\wedge f(1)\big)\\
&=& f\big(f(0)\vee(x\wedge f(1))\big) = f\big(\med(f(0),x,f(1))\big)\\
&=& \med(f(0),x,f(1)),
\end{eqnarray*}
which shows that $f$ is a polynomial function.
\end{proof}

\begin{proposition}\label{cor:HorizontalPoly2}
An order-preserving function $f\colon L^{n}\rightarrow L$ is a polynomial function if and only if $f$ satisfies the dual of (\ref{eq:deltaFG}),
and (\ref{eq:homogeneity}), and (\ref{eq:horizontal}).
\end{proposition}

\begin{proof}
By Lemmas~\ref{wlpHomogeneous} and \ref{Hor-Min-Hom} it follows that the conditions are necessary.

To prove the converse claim, we make use of Lemma~\ref{lemma:strSIdem}. Let $C$ be the class of order-preserving functions $f\colon L^n\to L$
$(n\geqslant 1)$ satisfying the conditions of the proposition. By Proposition~\ref{prop:unaryHorizontalPoly}, $C$ satisfies condition $(i)$ of
Lemma~\ref{lemma:strSIdem}. To complete the proof, it is enough to show that (\ref{eq:homogeneity}) and (\ref{eq:horizontal}) are preserved
under substituting constants for $n-1$ variables, since then condition $(ii)$ of Lemma~\ref{lemma:strSIdem} will be also fulfilled. Thus, take
$f\colon L^{n}\rightarrow L$ satisfying (\ref{eq:homogeneity}) and (\ref{eq:horizontal}).

Let $k\in [n]$ and $\vect{a}\in L^n$. To see that $f_{-k}^{\vect{a}}$ satisfies (\ref{eq:homogeneity}), just note that if $x\in L$ and $c\in
[f_{-k}^{\vect{a}}(0), f_{-k}^{\vect{a}}(1)]$ then,
\begin{eqnarray*}
f_{-k}^{\vect{a}}(x\wedge c)&=& f_{-k}^{\vect{a}\wedge\overline{c}}(x\wedge c)\vee f_{-k}^{[\vect{a}]_c}(0)\qquad \mbox{(by (\ref{eq:horizontal}))}\\
&=& (f_{-k}^{\vect{a}}(x)\wedge{c})\vee f_{-k}^{[\vect{a}]_c}(0)\qquad \mbox{(by (\ref{eq:homogeneity}))}\\
&=& f_{-k}^{\vect{a}}(x)\wedge{c}.
\end{eqnarray*}
To see that $f_{-k}^{\vect{a}}$ satisfies (\ref{eq:horizontal}), let $x\in L$ and $c\in [f_{-k}^{\vect{a}}(0), f_{-k}^{\vect{a}}(1)]$. We
clearly have $f(\vect{a}_k^x\wedge \overline{c})\leqslant f_{-k}^{\vect{a}}(x\wedge c)$ and $f([\vect{a}_k^x]_c)\leqslant
f_{-k}^{\vect{a}}([x]_c)$. Hence, by (\ref{eq:horizontal}), we get
$$
f_{-k}^{\vect{a}}(x)=f(\vect{a}_k^x\wedge \overline{c})\vee f([\vect{a}_k^x]_c)\leqslant f_{-k}^{\vect{a}}(x\wedge c)\vee
f_{-k}^{\vect{a}}([x]_c)\leqslant f_{-k}^{\vect{a}}(x).
$$
Since the above holds for every $x\in L$, we have that $f_{-k}^{\vect{a}}$ satisfies (\ref{eq:horizontal}).
\end{proof}


\begin{proposition}\label{cor:HorizontalPoly}
An order-preserving function $f\colon L^{n}\rightarrow L$ is a polynomial function if and only if $f$ satisfies (\ref{eq:deltaFG}) and its dual,
(\ref{eq:horizontal}) and its dual, and (\ref{eq:idempotency}).
\end{proposition}

\begin{proof}
By Lemmas~\ref{wlpHomogeneous}, \ref{lemma:MinMaxIdem}, and \ref{Hor-Min-Hom}, it follows that the conditions are necessary.

To prove the converse claim, we make use of Lemma~\ref{lemma:strSIdem}. Let $C$ be the class of order-preserving functions $f\colon L^n\to L$
$(n\geqslant 1)$ satisfying the conditions of the proposition. By Proposition~\ref{prop:unaryHorizontalPoly}, $C$ satisfies condition $(i)$ of
Lemma~\ref{lemma:strSIdem}. To complete the proof, it is enough to show that (\ref{eq:horizontal}), its dual, and (\ref{eq:idempotency}) are
preserved under substituting constants for $n-1$ variables, since then condition $(ii)$ of Lemma~\ref{lemma:strSIdem} will be also fulfilled.
Thus, take $f\colon L^{n}\rightarrow L$ satisfying (\ref{eq:horizontal}), its dual, and (\ref{eq:idempotency}).

Let $k\in [n]$ and $\vect{a}\in L^n$. To see that $f_{-k}^{\vect{a}}$ satisfies (\ref{eq:idempotency}), just note that if $c\in
[f_{-k}^{\vect{a}}(0), f_{-k}^{\vect{a}}(1)]$ then, by (\ref{eq:horizontal}) and its dual,
$$
f_{-k}^{\vect{a}}(c)=f_{-k}^{\vect{a}\wedge\overline{c}}(c)\vee f_{-k}^{[\vect{a}]_c}(0)\leqslant c\leqslant
f_{-k}^{\vect{a}\vee\overline{c}}(c)\wedge f_{-k}^{[\vect{a}]^c}(1) =f_{-k}^{\vect{a}}(c).
$$
The proof that $f_{-k}^{\vect{a}}$ satisfies (\ref{eq:horizontal}) follows exactly the same steps as in the proof of Proposition
\ref{cor:HorizontalPoly2}.
 The dual claim follows similarly.
\end{proof}

We can now provide a proof of the Main Theorem.

\begin{proof}[\bf Proof of the Main Theorem]
The equivalences $(i)\Leftrightarrow (ii)\Leftrightarrow (iii)$ are given by Propositions~\ref{Theorem:MAR} and \ref{cor:StrIdemWLP}. The
equivalence $(i)\Leftrightarrow (iv)$ follows from
Lemmas~\ref{wlpHomogeneous} and \ref{Hor-Min-Hom}, and Proposition \ref{cor:HorizontalPoly2}.
The equivalences $(i)\Leftrightarrow (v)$ and $(i)\Leftrightarrow (vi)$ follow from
Propositions \ref{cor:HorizontalPoly2} and
\ref{cor:HorizontalPoly}, respectively.
\end{proof}

\section{Concluding remarks}

By the equivalence $(i)\Leftrightarrow (iii)$, polynomial functions $f\colon L^{n}\rightarrow L$ with $f (\overline{0}) = 0$ and $f
(\overline{1}) = 1$ coincide exactly with those for which $\ran(f)=L$. These are referred to as \emph{discrete Sugeno integrals} and were
studied in \cite{Mar00c} where equivalence $(i)\Leftrightarrow (iv)$ of the Main Theorem was established for this particular case when $L$ is an
interval of the real line. Also, the implication $(i)\Rightarrow (v)$ of the Main Theorem reduces to that established by Benvenuti, Mesiar, and
Vivona \cite{BenMesViv02} when $L$ is an interval of the real line, since in this case the dual of (\ref{eq:deltaFG}) becomes redundant.
Condition (\ref{eq:homogeneity}) and its dual, when strengthened to all $c\in L$, are referred to as \emph{$\wedge$-homogeneity} and
\emph{$\vee$-homogeneity}, respectively; see \cite{GraMarMesPap09}. These were used by Fodor and Roubens~\cite{FodRou95} to axiomatize certain
classes of aggregation functions over the reals.

Recall that the property of being order-preserving is a consequence of all except assertion $(iv)$ of the Main Theorem. Also, given the nature
of statements $(iii)$--$(vi)$, it is natural to ask whether the equivalences between these and $(i)$ continue to hold over non-distributive
lattices. The reader can easily verify that (already for unary polynomial functions) this is not the case.

\end{document}